\documentclass{article}

\usepackage[english]{babel}

\usepackage[a4paper,top=2cm,bottom=2cm,left=3cm,right=3cm,marginparwidth=1.75cm]{geometry}

\usepackage{amsmath}
\usepackage{amssymb}
\usepackage{amsthm}
\usepackage{mathtools}
\usepackage{graphicx}
\usepackage[colorlinks=true, allcolors=blue]{hyperref}
\usepackage{relsize}
\newtheorem{theorem}{Theorem}

\newtheorem{conjecture}{Conjecture}
\newtheorem{lemma}{Lemma}
\newtheorem{definition}{Definition}
\usepackage{amsfonts}
\usepackage{tikz}
\usetikzlibrary{fit,shapes.geometric}
\usepackage{float}
\usepackage{caption}
\usepackage{subcaption}
\usepackage{booktabs} \usepackage{hhline}
\usepackage{comment}
\usepackage{todonotes}
\includecomment{comment}
\usetikzlibrary{shapes,arrows,positioning,calc}
\usepackage[numbers]{natbib}

\newcommand*\samethanks[1][\value{footnote}]{\footnotemark[#1]} 

\tikzset{every node/.style={circle}}
\tikzstyle{g0}=[scale=0.5,fill,outer sep=0]
\tikzstyle{g1}=[draw,inner sep=0,outer sep=0]

\title{Bounds on the Inducibility of Double Loop Graphs}
\author{S.~Y. Chan\thanks{Deakin University, Geelong, Australia, School of Information Technology, Faculty of Science Engineering \& Built Environment, \underline{Australia}}        
    \and    
    K. Morgan\samethanks
  
    \and
    J. Ugon\samethanks[1]}

\date{}

\begin{document}
\maketitle
\vspace{-6mm}
\begin{abstract}
In the area of extremal graph theory, there exists a problem that investigates the maximum induced density of a $k$-vertex graph $H$ in any $n$-vertex graph $G$. This is known as the problem of \emph{inducibility} that was first introduced by Pippenger and Golumbic in 1975. In this paper, we  give a new upper bound for the inducibility for a family of \emph{Double Loop Graphs} of order $k$. The upper bound obtained for order $k=5$ is within a factor of 0.964506 of the exact inducibility, and the upper bound obtained for $k=6$ is within a factor of 3 of the best known lower bound. 

\end{abstract}

\section{Introduction}

A problem in extremal graph theory investigates the maximum induced density of a $k$-vertex graph $H$ in any $n$-vertex graph $G$. The asymptotic behaviour of the limit of this maximum induced density of $H$ in different classes of graphs have been studied. This relates to the problem of inducibility that was first introduced by \citet{pippenger75} in 1975, which was defined as follows:

Given a graph $H$ of order $k$, we count the maximum number of $k$-vertex subsets in a graph G of order $n \geq k$, that induce a copy of $H$. Let $\mathcal{P}(H,G)$ denote the number of induced copies $H$ (up to isomorphism) in $G$, where $G$ and $H$ are simple graphs. This number lies between 0 and $\binom{n}{k}$. Since we are only interested in the limit as $n\rightarrow \infty$, we normalise $\mathcal{P}(H,G)$ by defining the \textit{inducibility} of $H$

\begin{equation}
\mathcal{I}(H,G)=\dfrac{\mathcal{P}(H,G)}{\binom{n}{k}}.   
\end{equation}

Let $\mathcal{P}(H, n)$ denote the maximum number of induced copies of $H$ taken over all graphs of order $n$. Again, we normalise this by setting
\begin{equation}
\mathcal{I}(H ,n) = \dfrac{\mathcal{P}(H,n)}{\binom{n}{k}}. 
\end{equation}
Note that
\begin{equation}
0 \leq  \mathcal{I}(H,G), \mathcal{I}(H, n) \leq 1.     
\end{equation}

Pippenger and Golumbic showed that for any graph $H$, the sequence $\mathcal{I}(H,n)$ is non-increasing and bounded below by 0 and approaches a limit as $n \rightarrow \infty$. Thus, the inducibility of $H$ is defined as

\begin{equation}
\mathcal{I}(H)=\lim_{n\rightarrow \infty} \mathcal{I}(H,n).  
\end{equation} 

The exact inducibility $\mathrm{I}(H)$ of only a handful of graphs and graph classes is currently known. This includes some graphs of small order \cite{balogh2016,evenzohar2015,hirst2011} and complete multipartite graphs \cite{bollobas1976,brownsidorenko1994}. Pippenger and Golumbic gave an initial lower bound on the inducibility for all $k$-vertex graphs,

\begin{equation}
\mathcal{I}(H)\geq \dfrac{k!}{k^{k}-k} \geq \dfrac{(2\pi k)^{\frac{1}{2}}}{e^{k}}.   
\label{eq4.5}
\end{equation}  

This lower bound was obtained using a \emph{balanced iterated blow-up of graph $H$}, a graph obtained by an iterative process of replacing each vertex of $G$ with copies of $H$. This lower bound that was obtained using this method of construction was later proven to be tight for any chosen graph $H$ by \citet{fox2021}. The authors showed, for a random graph $H$, that the maximum number of induced copies of $H$ in $G$ is best obtained by the balanced iterated blow-ups of $H$.

Trivially, we note that the complete graph $K_{n}$ and its complement $\overline{K_{n}}$ have inducibility 1. The inducibility $\mathcal{I}(H) = \mathcal{I}(\overline{H})$, thus, one need only consider one graph from each complementary pair to obtain the inducibility. The exact inducibilities of graphs on less than 4 vertices are known, except for the self-complementary path on 4 vertices ($P_{4}$). Table \ref{known inducibility} gives some results and constructions from \citet{exoo1986}, giving the known exact inducibility of graphs of small order and their respective extremal constructions.

Pippenger and Golumbic provided results on some families of graphs in \cite{pippenger75}. For cycles $C_{k}$, the authors conjectured the following lower bound and showed that this lower bound is tight for this family of graphs.

\begin{conjecture}[\citet{pippenger75}]
For $k\geq5$, $\mathcal{I}(C_{k})=\dfrac{k!}{k^{k}-k}$.
\label{pip conjecture}
\end{conjecture}

The condition where $k\geq 5$ is necessary, as $C_{3} = K_{3}$ and $C_{4} = K_{2,2}$ and the inducibility of $K_{n}$ and $K_{n,n}$ are precisely known. However, Conjecture \ref{pip conjecture} still remains unsolved \cite{hefetz2017}. Further, the authors also showed that

\begin{equation}
\mathcal{P}(C_{k},n)\leq \dfrac{2n}{k} \cdot \mathlarger(\dfrac{n-1}{k-1}\mathlarger)^{k-1}    
\end{equation}
such that
\begin{equation}
\mathcal{I}(C_{k})\leq 2e \cdot \dfrac{k!}{k^{k}}. 
\label{eq:bound1}
\end{equation}

Equation \eqref{eq:bound1} decreases the gap between the upper bound and lower bound of $\mathcal{I}(C_{k})$ with a multiplicative factor of $2e$. This upper bound was later improved by Hefetz and Tyomkyn (Theorem 4.0.2) in \cite{hefetz2017} for all cycles of size $k\geq 6$. The bound by Hefetz and Tyomkyn improved the gap to a factor of $\frac{128}{81}e$. Note that the exact inducibility for cycles of size $k = 5$ was solved by \citet{balogh2016}, by using a nested blow-up of $C_{5}$, to obtain $\mathcal{I}(C_{5}) = \frac{1}{26}$. 

\begin{theorem}[\citet{kral2019}]
Let $G$ be a graph of size $n$ and $C_{k}$ be a cycle of length $k\geq6$. Then $\mathcal{P}(C_{k},n)\leq \dfrac{2n^{k}}{k^{k}}$.
\label{theorem1}
\end{theorem}

Theorem \ref{theorem1} improves the upper bound for $\mathcal{I}(C_{k})$ further closing the gap between the upper and lower bounds by a multiplicative factor of 2 as follows

\begin{equation}
    \dfrac{k!}{k^{k}-k}\leq \mathcal{I}(C_{k})\leq (2+o(1))\cdot \dfrac{k!}{k^{k}}.
\end{equation}

Pippenger and Golumbic also computed the inducibility of $K_{\rho,\rho}$ and $K_{\rho,\rho+1}$. Note that for simplicity, we rewrite their original theorem (Theorem 10 from \cite{pippenger75}) as 

\begin{theorem}[\citet{pippenger75}]
For all $n\geq 2\rho \geq 3$: 
The maximum number of induced subgraphs isomorphic to the complete bipartite graph $K_{\rho,\rho}$ in any graph of order $n$ is 
\[ 
\mathcal{P}(K_{\rho,\rho},n)= \binom{\lceil{\frac{n}{2}}\rceil}{\rho} \binom{\lfloor{\frac{n}{2}}\rfloor}{\rho},
\]
and the inducibility is given by 
\[
\mathcal{P}(K_{\rho,\rho})=\dfrac{\binom{2\rho}{\rho}}{2^{2\rho}}.
\]
Similarly, for the complete bipartite graph $K_{\rho,\rho+1}$, where $n\geq 2\rho +1$, they have
\[ 
\mathcal{P}(K_{\rho,\rho+1},n)= \binom{\lceil{\frac{n}{2}}\rceil}{\rho} \binom{\lfloor{\frac{n}{2}}\rfloor}{\rho+1}+\binom{\lfloor{\frac{n}{2}}\rfloor}{\rho}\binom{\lceil{\frac{n}{2}}\rceil}{\rho}, 
\]
and
\[
\mathcal{P}(K_{\rho,\rho+1})=\dfrac{\binom{2\rho+1}{\rho}}{2^{2\rho}}.
\]
\end{theorem}

\citet{pippenger75} also provided a bound for some operations on graphs. Take a graph $G$ of order $k$ and $G'$ of order $k'$. Then the \emph{conjunction} $G\wedge G'$ is a single graph that consists of the disjoint graphs $G$ and $G'$. The \emph{disjunction} $G \vee G'$ is a connected graph that contains the disjoint graphs $G$ and $G'$ but with all possible edges added between $G$ and $G'$. The authors gave the following bound:

\begin{theorem}[Theorem 6~\citet{pippenger75}]
Let $G$ be a graph of order $k$ and $G'$ a graph of order $k'$,
\begin{equation}
\begin{rcases}
\mathcal{I}(G\wedge G')\\
\mathcal{I}(G\vee G')
\end{rcases}
\geq \dfrac{(k+k')!k^{k}k'^{k'}}{k!k'!(k+k')^{k+k'}}\mathcal{I}(G)\mathcal{I}(G')
\end{equation}
\label{theorem6pip}
\end{theorem} 

Some recent work \cite{hatami2014} gives sharp asymptotic results for larger graphs for $k\geq 5$. Razborov's theory of flag algebra \cite{razborov2007} has been a useful tool for proving results of inducibility of larger graphs. The case of $\mathcal{I}(P_{4})$ remains an open problem, however. Currently, the best known upper bound \cite{razborov2007} and lower bound \cite{evenzohar2015} of $\mathcal{I}(P_{4})$ have been found using flag algebra and the flagmatic software by \citet{vaughan2012}.

\begin{table}[H]
    \centering
    \begin{tabular}{|c c c c c c|}
    \hline
    $H$ & & $\Bar{H}$ & & $\mathcal{I}(H)$ & Extremal Construction \\
    \hhline{|= = = = = =|}
    $K_{3}=C_{3}$ & 
    \begin{tikzpicture}
        [scale=.4,auto=left,every node/.style={circle,fill=black, minimum size=1em}]
        \node[scale=0.6] (n1) at (0,0) {};
        \node[scale=0.6] (n2) at (2,0) {};
        \node[scale=0.6] (n3) at (1,1) {};
        \draw (n1) -- (n2);
        \draw (n2) -- (n3);
        \draw (n1) -- (n3);
        \end{tikzpicture}
        & $\overline{K_{3}}/A_{3}$ & 
    \begin{tikzpicture}
        [scale=.4,auto=left,every node/.style={circle,fill=black, minimum size=1em}]
        \node[scale=0.6] (n1) at (0,0) {};
        \node[scale=0.6] (n2) at (2,0) {};
        \node[scale=0.6] (n3) at (1,1) {};
        \end{tikzpicture} & 1 & A Complete Graph \\
    $P_{3}=K_{1,2}$ & 
    \begin{tikzpicture}
        [scale=.4,auto=left,every node/.style={circle,fill=black, minimum size=1em}]
        \node[scale=0.6] (n1) at (0,0) {};
        \node[scale=0.6] (n2) at (2,0) {};
        \node[scale=0.6] (n3) at (1,1) {};
        \draw (n1) -- (n2);
        \draw (n1) -- (n3);
        \end{tikzpicture} 
        & $\overline{P_{3}}$ &  
    \begin{tikzpicture}
        [scale=.4,auto=left,every node/.style={circle,fill=black, minimum size=1em}]
        \node[scale=0.6] (n1) at (0,0) {};
        \node[scale=0.6] (n2) at (2,0) {};
        \node[scale=0.6] (n3) at (1,1) {};
        \draw (n2) -- (n3);
        \end{tikzpicture} & $\frac{3}{4}$ & A Complete Bipartite Graph \\
    $K_{4}$ &
    \begin{tikzpicture}
        [scale=.4,auto=left,every node/.style={circle,fill=black, minimum size=1em}]
        \node[scale=0.6] (n1) at (0,0) {};
        \node[scale=0.6] (n2) at (1.5,0) {};
        \node[scale=0.6] (n3) at (1.5,1.5) {};
        \node[scale=0.6] (n4) at (0,1.5) {};
        \draw (n1) -- (n2);
        \draw (n2) -- (n3);
        \draw (n1) -- (n3);
        \draw (n1) -- (n4);
        \draw (n2) -- (n4);
        \draw (n3) -- (n4);
        \end{tikzpicture} 
        & $\overline{K_{4}}/A_{4}$ &
    \begin{tikzpicture}
        [scale=.4,auto=left,every node/.style={circle,fill=black, minimum size=1em}]
        \node[scale=0.6] (n1) at (0,0) {};
        \node[scale=0.6] (n2) at (1.5,0) {};
        \node[scale=0.6] (n3) at (1.5,1.5) {};
        \node[scale=0.6] (n4) at (0,1.5) {};
        \end{tikzpicture} 
        & 1 & A Complete Graph \\
    $S_{4}$ & 
    \begin{tikzpicture}
        [scale=.4,auto=left,every node/.style={circle,fill=black, minimum size=1em}]
        \node[scale=0.6] (n1) at (0,0) {};
        \node[scale=0.6] (n2) at (1.5,0) {};
        \node[scale=0.6] (n3) at (1.5,1.5) {};
        \node[scale=0.6] (n4) at (0,1.5) {};
        \draw (n1) -- (n2);
        \draw (n1) -- (n3);
        \draw (n1) -- (n4);
        \end{tikzpicture} 
        & $\overline{S_{4}}$ &
    \begin{tikzpicture}
        [scale=.4,auto=left,every node/.style={circle,fill=black, minimum size=1em}]
        \node[scale=0.6] (n1) at (0,0) {};
        \node[scale=0.6] (n2) at (1.5,0) {};
        \node[scale=0.6] (n3) at (1.5,1.5) {};
        \node[scale=0.6] (n4) at (0,1.5) {};
        \draw (n2) -- (n3);
        \draw (n2) -- (n4);
        \draw (n3) -- (n4);
        \end{tikzpicture}& $\frac{1}{2}$ & A Complete Bipartite Graph \\
    $C_{4}=K_{2,2}$ &
    \begin{tikzpicture}
        [scale=.4,auto=left,every node/.style={circle,fill=black, minimum size=1em}]
        \node[scale=0.6] (n1) at (0,0) {};
        \node[scale=0.6] (n2) at (1.5,0) {};
        \node[scale=0.6] (n3) at (1.5,1.5) {};
        \node[scale=0.6] (n4) at (0,1.5) {};
        \draw (n1) -- (n2);
        \draw (n2) -- (n3);
        \draw (n1) -- (n4);
        \draw (n3) -- (n4);
    \end{tikzpicture} & $\overline{C_{4}}$ &
    \begin{tikzpicture}
        [scale=.4,auto=left,every node/.style={circle,fill=black, minimum size=1em}]
        \node[scale=0.6] (n1) at (0,0) {};
        \node[scale=0.6] (n2) at (1.5,0) {};
        \node[scale=0.6] (n3) at (1.5,1.5) {};
        \node[scale=0.6] (n4) at (0,1.5) {};
        \draw (n1) -- (n3);
        \draw (n2) -- (n4);
        \end{tikzpicture}& $\frac{3}{8}$ & A Complete Bipartite Graph \\
    $K_{paw}$ &
    \begin{tikzpicture}
        [scale=.4,auto=left,every node/.style={circle,fill=black, minimum size=1em}]
        \node[scale=0.6] (n1) at (0,0) {};
        \node[scale=0.6] (n2) at (1.5,0) {};
        \node[scale=0.6] (n3) at (1.5,1.5) {};
        \node[scale=0.6] (n4) at (0,1.5) {};
        \draw (n1) -- (n2);
        \draw (n2) -- (n3);
        \draw (n1) -- (n3);
        \draw (n2) -- (n4);
        \end{tikzpicture}& $\overline{K_{paw}}$ & 
    \begin{tikzpicture}
        [scale=.4,auto=left,every node/.style={circle,fill=black, minimum size=1em}]
        \node[scale=0.6] (n1) at (0,0) {};
        \node[scale=0.6] (n2) at (1.5,0) {};
        \node[scale=0.6] (n3) at (1.5,1.5) {};
        \node[scale=0.6] (n4) at (0,1.5) {};
        \draw (n1) -- (n4);
        \draw (n3) -- (n4);
        \end{tikzpicture} & $\frac{3}{8}$ & Two Disjoint Complete Bipartite Graphs \\
    $K_{1,1,2}$ &
    \begin{tikzpicture}
        [scale=.4,auto=left,every node/.style={circle,fill=black, minimum size=1em}]
        \node[scale=0.6] (n1) at (0,0) {};
        \node[scale=0.6] (n2) at (1.5,0) {};
        \node[scale=0.6] (n3) at (1.5,1.5) {};
        \node[scale=0.6] (n4) at (0,1.5) {};
        \draw (n2) -- (n3);
        \draw (n1) -- (n3);
        \draw (n1) -- (n4);
        \draw (n2) -- (n4);
        \draw (n3) -- (n4);
        \end{tikzpicture}& $\overline{K_{1,1,2}}$ &
            \begin{tikzpicture}
        [scale=.4,auto=left,every node/.style={circle,fill=black, minimum size=1em}]
        \node[scale=0.6] (n1) at (0,0) {};
        \node[scale=0.6] (n2) at (1.5,0) {};
        \node[scale=0.6] (n3) at (1.5,1.5) {};
        \node[scale=0.6] (n4) at (0,1.5) {};
        \draw (n1) -- (n2);
        \end{tikzpicture}& $\frac{72}{125}$ & A Complete 5-equipartite Graph \\
    $P_{4}$ &
    \begin{tikzpicture}
        [scale=.4,auto=left,every node/.style={circle,fill=black, minimum size=1em}]
        \node[scale=0.6] (n1) at (0,0) {};
        \node[scale=0.6] (n2) at (1.5,0) {};
        \node[scale=0.6] (n3) at (1.5,1.5) {};
        \node[scale=0.6] (n4) at (0,1.5) {};
        \draw (n1) -- (n2);
        \draw (n2) -- (n3);
        \draw (n1) -- (n4);
        \end{tikzpicture} & $\overline{P_{4}}$ & 
    \begin{tikzpicture}
        [scale=.4,auto=left,every node/.style={circle,fill=black, minimum size=1em}]
        \node[scale=0.6] (n1) at (0,0) {};
        \node[scale=0.6] (n2) at (1.5,0) {};
        \node[scale=0.6] (n3) at (1.5,1.5) {};
        \node[scale=0.6] (n4) at (0,1.5) {};
        \draw (n1) -- (n3);
        \draw (n2) -- (n4);
        \draw (n3) -- (n4);
        \end{tikzpicture}& $?$ & Undetermined \\
    \hline
    \end{tabular}
    \caption{The known inducibility of graphs of order 3 and 4 \cite{evenzohar2015}.}
    \label{known inducibility}
\end{table}

\subsection{Some Bounds and Construction Methods}
The inducibility of some family of graphs have been determined, i.e., complete graph $K_{k}$ \cite{pippenger75}, complete bipartite graph $K_{\rho,\rho}$,$K_{\rho,\rho +1}$ \cite{pippenger75} and 
complete multipartite graphs $K_{k,\ldots,k}$ \cite{brownsidorenko1994}. Inducibility of various graph classes have also been studied, such as $d$-ary trees \cite{czabarka2020darytrees}, rooted trees \cite{dossouolory2018}, directed paths \cite{choi2020}, oriented stars \cite{huang2014}, net graphs \cite{blumenthalNet2021} and random Cayley graphs \cite{fox2021}. The precise inducibility of paths of order $k\geq4$ or cycles of order $k\geq 6$ are still unknown \cite{fox2015}.

In the case of $P_{4}$, Exoo's construction \cite{exoo1986} gave an initial lower bound of $\frac{960}{4877}\approx 0.1968$ and an upper bound of $\frac{1}{3}$. The bounds were attained by the of blow-up of the \emph{Paley Graph}. The upper bound was first improved to $\approx 0.2064$ by \citet{hirst2011} using semi-definite programming methods. \citet{vaughan2012} found the currently best known upper bound of $\approx 0.204513$ using Flag algebra method \cite{razborov2007}. \citet{evenzohar2015} improved the lower bound to $\frac{1173}{5824}\approx 0.2014$ using a combination of the Erd\"{o}s conjecture \cite{erdos1969} and Thomason's construction \cite{thomason1997}. Table \ref{tab:summary P4} summarises the results. 

\begin{table}[htb]
    \centering
    \begin{tabular}{|p{0.2\linewidth} p{0.25\linewidth} p{0.25\linewidth} p{0.25\linewidth}|}
    \hline
    Name \& Author(s) & Lower Bound & Upper Bound & Construction \\
    \hhline{|= = = =|}
    Pippenger \& Golumbic~\cite{pippenger75}, Exoo \cite{exoo1986} & $\mathcal{I}(P_{k})\geq \dfrac{k!}{(k+1)^{k-1}-1}$ & $\mathcal{I}(P_{k})\leq \dfrac{k!}{2(k-1)^{k-1}}$ & Nested blow-up of graphs. \\
    \hline
    Exoo \cite{exoo1986} & 
    $ \mathcal{I}(P_{4}) \geq \dfrac{960}{4877}  
    \approx 0.1968$
    & 
    $\mathcal{I}(P_{4}) \leq \dfrac{1}{3} \approx 0.3333$
    & Paley Graphs of order $p \cong 1\pmod{4}$. \\
    \hline
    Hirst \cite{hirst2011} & \emph{unknown} & $\mathcal{I}(P_{4}) \leq 0.2064$ & Semi-definite method. \\
    \hline
    Vaughan \cite{vaughan2012} & \emph{unknown} & $\mathcal{I}(P_{4})\approx 0.204513$ & Flag algebra calculus. \\
    \hline
    Even-Zohar \cite{evenzohar2015} &
    $ \mathcal{I}(P_{4}) \geq \dfrac{1173}{5824} 
     \approx 0.2014 $
    & \emph{unknown} & Combination of Erd\"{o}s conjecture \cite{erdos1969} and Thomason's construction \cite{thomason1997}. \\
    \hline
    \end{tabular}
    \caption{Summary of known bounds and constructions for $P_{4}$.}
    \label{tab:summary P4}
\end{table}

The exact inducibility of only some 5-vertex graphs has been determined. The exact inducibility results for some of the 5-vertex graphs were found using nested blow-ups \cite{evenzohar2015}. The authors also adapted the Erd\"{o}s-R\'{e}nyi random graph as the extremal construction for some of the 5-vertex graphs. For further readings on the inducibility of 5-vertex graphs and the extremal construction, the reader is referred to \cite{evenzohar2015}.

It is interesting to note that random graphs and probabilistic methods have been one of the chosen tools in solving the inducibility problem for some classes of graphs \cite{evenzohar2015,fox2021,yuster2019}. Currently, the best known validation tool for extremal combinatorics is the flagmatic \cite{vaughan2012} software based on Flag algebra \cite{razborov2007}. However, the exact inducibility and extremal construction of some graph classes are not yet determined, including $P_{4}$ and some of the graphs on 5-vertices.

\subsection{Motivation}
In this paper, we will investigate the inducibility of a family of \emph{circulant} graphs. These are special cases of \emph{Cayley graphs}. 

\begin{definition}
Let $G$ be a finite group and $S$ be a subgroup of $G$ that contains non-identity elements. The \emph{Cayley} graph  $X=\mathcal{C}(G,S)$ is an undirected graph with vertex set $V(X)=G$ and edge set $E(X)=\{g,gh : g\in G, h\in S\}$, such that $S$ is symmetric and generates $G$.   
\end{definition}

Cayley graphs contain symmetrical properties 
and are  widely studied in many disciplines. In  areas of computer science, Cayley graphs have been used in quantum error correction and error-correcting codes 
in quantum systems \cite{couvreur13,tomic13,vandermolen22,zemor09}. Further, this family of graphs have also been studied in relation to Boolean functions \cite{bernasconi99,pourfaraj20,riera18}. In network topology, Cayley graphs are  tools for studying the impact on the reliability of a network, fault tolerance and effective routing (communication) within networks, where vertices represent processing elements and edges represent the communication lines in the network \cite{camelo14,gu18,heydemann97,hsieh05,jiang2019,mokhtar17,shahzamanian10,song11,vadapalli95,vadapalli96,xu17,zhao2020}. 

In biology, Cayley graphs are used in the areas of genomics, particularly in the effects of genome arrangements based on permutation of groups and graph convexity, since permutations are used to represent gene sequences in genomes and chromosomes \cite{cunha18,cunha19,egri14,francis20,li08,moulton12}. In the areas of chemistry, Cayley graphs have been used for tracking atoms as a result of a sequence of chemical reactions, in order to anticipate the occurrence of the product molecules \cite{hellmuth20,nojgaard21}. Chemical graphs represented by Cayley graphs are also used to study the \emph{Wiener Index} \cite{knor19}. 

Cayley graphs were originally first used in explaining the concept of abstract groups, and have been of particular interests amongst group theorists. One of the problems in this area utilises group theory to study the spectra of this family of graphs \cite{babai79,buser88,qin09}. Another problem in this area investigates the relationship between finite groups on surfaces and Cayley graphs \cite{bonfert11,bozejko06,tucker83}.

Circulant graphs are a special subset of Cayley graphs. They are Cayley graphs for the cyclic group $\mathbb{Z}_{n}$ \cite{godsil12}. A 3-jump circulant graph on 12 vertices is an undirected Cayley graph. Like Cayley graphs, circulant graphs also have symmetrical properties that enable rerouting of networks, network security and even used in peer-to-peer networking \cite{kaufmann09,li98,liu09}. This family of graphs has a nice cyclic property that contains extra connections (graph edges). 

In networks, these circulant graphs are used in terminal rerouting, communication or broadcasting problems \cite{cai99,cai552,obradovic05,perez-roses22, romanov20}. For example, if there is a redundancy in a traffic network, the (extra) edges can prevent unnecessary long routes for information to travel between terminals.  Further, the family of circulant graphs are also studied in the shortest path problems \cite{cai552, Lu16}. The topology and metrics of circulant graphs are also of interest \cite{alspach79,imran12,munir16,muzychuk97,shparlinski06}. 




In most cases, circulant graphs have specified \emph{jump lengths}. These jump lengths can be constant or non-constant. One of the simplest families of circulant graphs are cycles. A cycle  of order $n$ is a 1-jump circulant graph where $DLG(a,b) = DLG(-1,1)$. Although, cycles are one of the highly studied families of graphs in the areas of inducibility, the exact inducibility of cycles of order $k\geq 6$ still remains unknown. In this paper, we are interested in 2-jump circulant graphs, or \emph{Double loop graphs}. 


\begin{definition}[Double loop graph \cite{boesch1984}] 
Let $a,b,n$ be integers with $0<a\neq b<\frac{n}{2}$; the vertices of the \emph{double loop graph} $DLG(a,b)$ are the integers modulo $n$, each vertex $i$ being joined to the four vertices $i\pm a \pmod{n}$, $i\pm b \pmod{n}$. 
\end{definition}

\begin{figure}[H]
    \centering
    \begin{tikzpicture}
    [scale=.7,auto=left,every node/.style={circle,fill=black}]
    \node (n1) at (0,0) {};
    \node (n2) at (2,0) {};
    \node (n3) at (-1,-1.5) {};
    \node (n4) at (3,-1.5) {};
    \node (n5) at (0,-3) {};
    \node (n6) at (2,-3) {};
    
    \draw (n1) edge (n2);
    \draw (n1) edge (n3);
    \draw (n2) edge (n3);
    \draw (n2) edge (n4);
    \draw (n3) edge (n5);
    \draw (n4) edge (n6);
    \draw (n5) edge (n6);
    \draw (n2) edge (n6);
    \draw (n3) edge (n6);
    \draw (n1) edge (n4);
    \draw (n1) edge (n5);
    \draw (n4) edge (n5);
    
    \end{tikzpicture}
    \caption{A double loop graph $DLG(1,2)$ on 6 vertices.}
    \label{fig:my_label}
\end{figure}
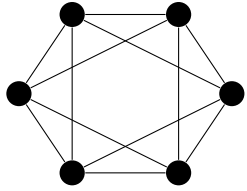

Double loop graphs are circulant graphs and so are of interest in the areas discussed above.  Work has been done of finding a minimum 2-terminal routing using such graphs \cite{cai552}. Other topics include cycle decomposition \cite{bogdanowicz15}, the shortest path problems \cite{gomez05}, the dispersability of circulant graphs \cite{joslin21} and optimal routing \cite{gomez07}.  In this paper, we give bounds on the inducibility of 2-jump circulant graphs, $DLG(1,2)$. 

\section{Notations and Definitions} 
In this section, we present some basic definitions and notations that are used in this paper. All graphs in this paper are simple.

A graph $G$ is a pair $(V,E)$, such that $V$ is the (finite) set of vertices and $E$ is a subset of all unordered pair of vertices. The \emph{order} of a graph is the number of vertices, whereas the \emph{size} of a graph is the number of edges. Let $u,v\in V(G)$, we say that $u$ is \emph{adjacent} to $v$ if there exists an edge $\{u,v\}\in E(G)$. We say that the edge $\{u, v\}$ is incident to vertices $u$ and $v$.

\begin{definition}
Let $G$ and $H$ be graphs. If there exists an edge-preserving bijective function $\phi:E(G)\rightarrow E(H)$, such that:
\[
\forall(u,v)\in V(G)\times V(G), (u,v)\in E(G) \leftrightarrow (\phi(u),\phi(v))\in E(H),
\]
then $G$ and $H$ are isomorphic, denoted $G\cong H$.
\end{definition}

\begin{definition}
Let $G$ and $G_{1}$ be graphs of order $n$ and $k$ respectively, where
$n\geq k$. We say that $G_{1}$ is a subgraph of $G$ if $V(G_{1})\subseteq V(G)$ and $E(G_{1})\subseteq E(G)$. The graph $G_{1}$ is an induced subgraph of $G$ if all the edges between the pairs of vertices in $V(G_{1})$ from $E$ are in $E(G_{1})$.
\end{definition}

\section{Proof for Double Loop Graph (2-jump Circulant Graph)}

In this section, we will give an upper bound on the inducibility of a family of double loop graphs. The family of double loop graphs contains symmetrical properties similar to cycles $C_{n}$, which makes it a natural choice of graphs. We first provide some necessary definitions that will be used in the proof. 

\begin{definition}[Chain graph $L_{k}$]
We recursively define a chain graph $L_{k}$ as follows. The graph $L_{0}$ is the triangle with vertices $v_{0},v_{1}$ and $v_{2}$. A chain graph $L_{k}$ is a chain of triangles formed by connecting $L_{k-1}$ to a new vertex $v_{k+2}$ to the edge $\{v_{k},v_{k+1}\}$, for $k>0$. 

\end{definition}

\begin{figure}[H]
    \centering
    \begin{tikzpicture}
    [scale=.7,auto=left,every node/.style={circle,fill=black}]
    \node[label=below:$v_{1}$] (n1) at (0,0) {};
    \node[label=160:$v_{0}$] (n2) at (-1,1) {};
    \node[label=above:$v_{2}$] (n3) at (1,1) {};
    \node[label=-45:$v_{3}$] (n4) at (2,0) {};
    \node[label=above:$v_{4}$] (n5) at (3,1) {};
    \node[label=below:$v_{5}$] (n6) at (4,0) {};
    \node[label=above:$v_{6}$] (n7) at (5,1) {};
    \node[draw=blue,fill=none,double,fit=(n1) (n2) (n3) ,inner sep=0pt,ellipse,label={[text=blue]215:$L_{0}$}] (tmp) {};
     \node[draw=red,fill=none,double,fit=(n1) (n2) (n3) (n4) ,inner sep=0.5pt,circle,label={[text=red]90:$L_{1}$}] (tmp) {};
    
    \draw (n1) edge (n2);
    \draw (n1) edge (n3);
    \draw (n2) edge (n3);
    \draw (n1) edge (n4);
    \draw (n3) edge (n4);
    \draw (n4) edge (n5);
    \draw (n3) edge (n5);
    \draw (n5) edge (n6);
    \draw (n4) edge (n6);
    \draw (n6) edge (n7);
    \draw (n5) edge (n7);
    \end{tikzpicture}
    \caption{An illustration of a chain graph $L_{4}$ on 7 vertices. $L_{0}$ is the triangle graph circled in blue, while the chain graph $L_{1}$ is circled in red.}
    \label{fig:chaingraph}
\end{figure}

\noindent
We provide the following definitions for a \emph{good tuple}, \emph{loopy $k$-tuple} and a \emph{good prefix}:
\begin{definition}[Good tuple]
We define a \textit{good tuple} of order $l>0$ recursively.  For $l\leq 3$, the ordered sequence of vertices (ordered tuple), $(v_{0}, v_{1},\ldots, v_{l-1})$ is a good tuple if it induces the complete graph $K_{l}$.  For any $l>3$,  the ordered tuple $(v_{0}, v_{1}, \ldots, v_{l-1})$ is a good tuple if it induces the chain graph $L_{l-3}$ and the ordered tuple $(v_{0}, v_{1},\ldots,v_{l-2})$ is a good tuple.
\end{definition}

\begin{definition}[Loopy $k$-tuple]
A loopy $k$-tuple is an ordered sequence of vertices $(v_{0}, v_{1}, \ldots, v_{k-1})$ which induces a copy of the double loop graph of order $k$, and for $l< k-1$, every ordered tuple $(v_{0}, v_{1}, \ldots ,v_{l})$ is a good tuple.    
\end{definition}

\begin{definition}[Good prefix]
An ordered tuple $(v_{0}, v_{1}, \ldots v_{l-1})$ of length $l\leq k$ is a good prefix, if it is a good tuple of length $l<k$, or a loopy $k$-tuple.
\end{definition} 

\noindent
We state the following theorem.
\begin{theorem}
Every graph on $n$ vertices contain at most $\dfrac{27n^{k}}{k^{k}}$ induced copies of the double loop graph $DLG(1,2)$ for $k\geq 5$.
\label{thm:DLG}
\end{theorem}

\begin{proof}
Suppose we have a graph $G$ of $n$, and an integer $k\geq 5$. We will count the number of loopy $k$-tuples in the graph $G$. We define the probability distribution $P(D)$ based on conditional probability for each loopy $k$-tuple $D=\left(v_{0},v_{1},\ldots,v_{k-1}\right)$ as

\begin{equation}
    P(D)=\prod^{k-1}_{i=0} P_{i}(v_{i}|v_{i-1}).
\end{equation}
\\
Let $n_{i}$ denote the number of choices for $v_{i}$ given $(v_{0}, \ldots, v_{i-1})$ is a good prefix.  We denote the probability $P_{i}(v_{i}|(v_{0},\ldots,v_{i-1}))=\frac{1}{n_{i}}$.  When $i=0$, we can choose any vertex and so $$P_{0}(v_{0})=\frac{1}{n_{0}}=\frac{1}{n}.$$ 

For $1\leq i\leq k$, the choice for $v_{i}$ must ensure that $(v_{0}, \ldots, v_{i-1})$ is a \emph{good prefix} and the probability for each $P_{i}(v_{i})$ is dependent on the choices of the vertices  $v_{0}, \ldots, v_{i-1}.$

When $i=1,2$, we have
    \begin{align*}
           P_{1}(v_{1})&=\frac{1}{n_{1}}=\frac{1}{deg(v_{0})},\\
   P_{2}(v_{2})&=\frac{1}{n_{2}}=\frac{1}{|N(v_{0})|\cap|N(v_{1})|}
    \end{align*}
where $N(v)$ denotes the neighbours of $v$.

The vertex $v_{k-1}$ can be chosen if and only if the vertex $v_{k-1}$ is adjacent to vertices $v_{k-2},v_{k-3},v_{0}$ and $v_{1}$, but not adjacent to any other vertex belonging to the ordered tuple. 

Recall that a \emph{good prefix} is an ordered tuple of length $l\leq k$.
If a maximal good prefix has length $k$, then it gives a loopy $k$-tuple. If a good tuple of length $l<k$ can not be further extended to a loopy $k$-tuple, then the good tuple  is a maximal good prefix. We state the following lemma for the probability distribution $P(D)$.

\begin{lemma}
The probability distribution $P(D)$ of all loopy $k$-tuples $D$ has a sum of at most 1.
\label{lemma:sum1}
\end{lemma}

Consider the vertices that induce a $DLG(1,2)$ such that the loopy $k$-tuple $D_{j}$ is of the sequence ($v_{j}v_{j+1}v_{j+2}\ldots v_{j+k-1}$) for $j=0,1,\ldots,k-1$, indices  $\bmod$ $k$. We want to show that

\begin{equation}
    \dfrac{k^{k}}{54n^{k}}\leq P(D_{0})+\ldots+P(D_{k-1}).
    \label{eq:sumbound}
\end{equation}
Each $D_{j}$ corresponds to 2 loopy $k$-tuples namely $v_{j}v_{j+1}v_{j+2}\ldots v_{j+k-1}$ and $v_{j+k-1}v_{j+k-2}v_{j+k-3}\ldots v_{j}$. Thus, we have a total of $2k$ loopy $k$-tuples for each induced $DLG(1,2)$.   Thus, \eqref{eq:sumbound} would imply that the sum of these $2k$ loopy $k$-tuples is at least $\dfrac{k^{k}}{27n^{k}}$. It follows from Lemma \ref{lemma:sum1} that there are at most $\dfrac{27n^{k}}{k^{k}}$ induced $DLG(1,2)$. To prove \eqref{eq:sumbound}, consider the AM-GM inequality, such that
\begin{equation}
    \dfrac{P(D_{0})+P(D_{1})+P(D_{2})+\ldots+P(D_{k-1})}{k}\geq \left(\prod^{k-1}_{j=0} P(D_{j})\right)^{\frac{1}{k}}.
    \label{AM-GM}
\end{equation}

Let $n_{i,j}$ be the number of choices for $v_{i}$ in $D_{j}$. Note, $n_{0,j}=n$ for all $j$. Recall that for $i=2,\ldots,k-2$, a vertex $v_{i}$ is chosen from the neighbours of vertex $v_{i-1}$ given that $v_{0}v_{1}v_{2}\ldots v_{i-1}$ is a good prefix. The last vertex $v_{k-1}$ is chosen if it is only adjacent to $v_{k-2},v_{1},v_{2}$ and $v_{0}$, but not adjacent to any other vertex that belongs in the good tuple. We take the right-hand side of the inequality \eqref{AM-GM} to get the following,

\begin{align}
    \left(\prod^{k-1}_{j=0} \dfrac{1}{P(D_{j})}\right)^{\frac{1}{k(k-1)}} &= \prod^{k-1}_{j=0}\left(n_{0,j}\frac{n_{1,j}}{3}\frac{n_{2,j}}{3}{n_{3,j}}\ldots n_{k-2,j} \frac{n_{k-1,j}}{3}\right)^{\frac{1}{k(k-1)}}
    \label{eq:right2}
\end{align}

Let $v'$ be a potential candidate for $v_{i}$ and consider the contributions of $v'$ to each of the $n_{i,j}$ term for each $j=0,1,\ldots,k-1$. Due to the symmetrical properties of $DLG(1,2)$, we need only consider the case for $j=0$.

We know that $n_{0,j}=n$ for all $j$ by definition. Similarly, we know that each tuple $D_{j}$ contributes to 2 loopy $k$-tuples. Now let $l>i$ be the smallest integer in $[0,k)$ such that $v'$ is adjacent to  $v_{i}$ and $v_{l}$, and $\{v_{i},v_{l}\}$ gives an edge that belongs to the ordered tuple $(v_{0}, v_{1}, \ldots, v_{k-1})$.

If $i=0$ and $l=1$, then $v'$ contributes to each $n_{1,0}$, $n_{2,0}$ and $n_{k-1,0}$, giving the terms $\frac{n_{1,j}}{3},\frac{n_{2,j}}{3}$ and $\frac{n_{k-1,j}}{3}$ from equation \eqref{eq:right2}. Thus, 
\begin{align}
    \left(\prod^{k-1}_{j=0} \dfrac{1}{P(D_{j})}\right)^{\frac{1}{k(k-1)}} 
    &=\left((2\cdot n\cdot(3^{3}))^{k}\prod^{k-1}_{j=0}\frac{n_{1,j}}{3}\frac{n_{2,j}}{3}n_{3.j}\cdots n_{k-2,j} \frac{n_{k-1,j}}{3}\right)^{\frac{1}{k(k-1)}}\\
       &= (54n)^{\frac{1}{k-1}}\left(\prod^{k-1}_{j=0}\frac{n_{1,j}}{3}\frac{n_{2,j}}{3}n_{3.j}\cdots n_{k-2,j} \frac{n_{k-1,j}}{3}\right)^{\frac{1}{k(k-1)}}\\
       \label{eq:amgmequality}
       \intertext{From equation \eqref{AM-GM},}
    \left(\prod^{k-1}_{j=0} \dfrac{1}{P(D_{j})}\right)^{\frac{1}{k(k-1)}}   &\leq \dfrac{(54n)^{\frac{1}{k-1}}}{k(k-1)} \left(\sum^{k-1}_{j=0} \left(\frac{n_{1,j}}{3}+\frac{n_{2,j}}{3}+{n_{3,j}}+\ldots+n_{k-2,j}+\frac{n_{k-1,j}}{3}\right )\right)
   \label{rightsum}
\end{align}
For all $i>0$, if $l=i+1$, $v'_{i}$ contributes to $n_{i+2,0}$. However, if $l>i+1$, then $v'$ does not contribute to the sum for any $j$. Considering all possible cases, $v'$ contributes at most 1 for every $j$. Thus, $v'$ adds at most $k$ to the total sum in \eqref{rightsum}. As there are $n$ vertices in the graph $G$, each vertex $v'\in V(G)$ contributes at most 1 to the sum in \eqref{rightsum}, making the whole sum at most $n\cdot k$. We derive the upper bound (Equation \ref{rightsum}) from Equation  \ref{eq:amgmequality} using the AM-GM inequality, and it follows that 

\begin{align}
    \left(\prod^{k-1}_{j=0} \dfrac{1}{P(D_{j})}\right)^{\frac{1}{k(k-1)}} &\leq \dfrac{(54n)^{\frac{1}{k-1}}}{k(k-1)}\cdot n\cdot k = \dfrac{(54n)^{\frac{1}{k-1}}\cdot n\cdot k}{k(k-1)}\\
    \intertext{Which then gives}
    \left(\prod^{k-1}_{j=0} \dfrac{1}{P(D_{j})}\right)^{\frac{1}{k}} &\leq \dfrac{54n^{k}}{(k-1)^{k-1}} \\
    \intertext{Thus,}
    \left(\prod^{k-1}_{j=0}P(D_{j})\right)^{\frac{1}{k}} &\geq \dfrac{(k-1)^{k-1}}{54n^{k}}.
\end{align}
\end{proof}

\subsection{Further Results for $k$}
In this section, we introduce constructions for the $DLG(1,2)$ graph to obtain a lower bound for some $k$. First, notice that when $k=5$, the $DLG(1,2)$ is the complete graph $K_{n}$ of order 5. Using the upper bound obtained in equation \eqref{eq:bound1}, we obtain the following lemma:

\begin{lemma}
The inducibility of $DLG(1,2)$ is given by $\mathcal{I}(DLG(1,2))\leq \dfrac{27k!}{k^{k}}$.
\label{dlglemma}
\end{lemma}

By Lemma \ref{dlglemma}, we obtain $\mathcal{I}(DLG(1,2))\leq 1\frac{23}{625}\approx 1.0368$ for $k=5$. This gives a close result to the known result for inducibility of complete graphs $\mathcal{I}(K_{5})=1$. 

For the case where $k=6$, we give a construction that gives a lower bound to the inducibility $\mathcal{I}(DLG(1,2))\geq\frac{10}{81}$. This result brings the gap between the upper bound and the lower bound to a factor of $3$. Our construction uses a partitioning method for $DLG(1,2)$ for $k=6$ and is as follows: Consider a subgraph $C_{4}$ of a $DLG(1,2)$.

We take 3 partitions $P_{1},P_{2},P_{3}$ such that each partition contains $n$ vertices. The partitions $P_{1}$ and $P_{2}$ give a complete bipartite graph $K_{n,n}$ and $P_{3}$ is the set of isolated vertices which are connected to all vertices in $P_{1}$ and $P_{2}$. From each partition, we pick two vertices. The graph induced by these six vertices is $DLG(1,2)$. Thus,
\begin{align}
    \mathcal{P}(DLG(1,2),n) &=\binom{n}{2}\cdot\binom{n}{2}\cdot\binom{n}{2} = \dfrac{n^{3}(n-1)^{3}}{8}.
    \label{countsdlg}
\end{align}
Now, from Equation \eqref{countsdlg}, we have:
\begin{align*}
   \dfrac{\dfrac{n^{3}(n-1)^{3}}{8}}{\binom{3n}{6}} &= \dfrac{n^{3}(n-1)^{3}}{8}\cdot\dfrac{6!}{3n(3n-1)(3n-2)(3n-3)(3n-4)(3n-5)}. 
 \end{align*}
 \noindent
Taking the limit as $N$ approaches $\infty$, we obtain the lower bound on the inducibility of $DLG(1,2)$
 \begin{align}
   \mathcal{I}(DLG(1,2)) &\approx \dfrac{6!\cdot n^{6}}{3^{6}n^{6}}= \dfrac{10}{81}.
     \label{dlgk6}
\end{align}

Recall that a graph $G$ and it's complement share the same inducibility, that is $\mathcal{I}(G)=\mathcal{I}({\overline{G}})$. The complement of the $DLG(1,2)$ on $k=6$ gives 3 disjoint $K_{2}$s. To compare our results with the bound using Theorem \ref{theorem6pip}, we construct $G$ and $G'$ as follows. Take two of the disjoint $K_{2}$s as $G$ and the other $K_{2}$ as a graph $G'$. We know that $\overline{G} = C_{4}$ and $\overline{G'}= N_{2}$. Take $k=4$ and $k'=2$, computing inducibility using the bound from Theorem \ref{theorem6pip} we have:
\begin{align*}
\mathcal{I}(G\wedge G') &\geq \dfrac{6!\cdot4^{4}\cdot2^{2}}{4!\cdot2!\cdot6^{6}}\cdot \frac{3}{8}\cdot 1  =\frac{10}{81}
\end{align*}
which gives the same result as \eqref{dlgk6}.

Now for the case $k=7$ for $DLG(1,2)$, we consider the complement $\overline{DLG(1,2)}$ which is precisely the cycle $C_{7}$. The known conjecture for cycles was given by~\citet{pippenger75}, where $\mathcal{I}(C_{k})\geq \dfrac{k!}{k^{k}-k}$ for $k\geq 5$. Further, ~\citet{kral2019} proved an upper bound for $k\geq 6$ where $\mathcal{I}(C_{k})\leq (2+o(1))\cdot\frac{k!}{k^{k}}$. The current best construction known for $C_{k}$ is the nested blow-up of a $C_{k}$. Thus, for the $\overline{DLG(1,2)}$ for $k=7$, the inducibility is calculated as:
\[
\mathcal{I}(C_{7}) = \mathcal{I}(DLG(1,2)) \geq \frac{5}{817}.
\]
Using the upper bound from our construction we obtain $\mathcal{I}{(DLG(1,2))}\leq 0.165237$, which gives a factor of 27 to our upper bound.

\section{Conclusion}

In this paper, we presented a proof for the upper bound on the inducibility of the family of double loop graphs $DLG(1,2)$ in a given graph $G$. For the case when $k=5$, we showed that our upper bound is indeed close to the known inducibility of $K_{5}$, where $\mathcal{I}(K_{5})=1$. We further showed the lower bound for the counts of $DLG(1,2)$ for $k=6$. We presented a construction for the $DLG(1,2)$. When compared to the bound given by \citet{pippenger75} for disjoint graphs, we showed that our construction is in fact the best possible construction that maximises the number of $DLG(1,2)$ in any given graph $G$. 

Future work include finding the lower bound for this class of graphs for $k\geq 8$. We may also consider finding the bounds on the inducibility for the family of graphs that have some symmetrical properties.

\bibliographystyle{plainnat}
\bibliography{main}

\end{document}